\documentclass[12pt]{amsart}

\usepackage{amsmath,amssymb,amsfonts,amscd,pictexwd,dcpic}

\newtheorem{theorem}{Theorem}[section]
\newtheorem{lemma}[theorem]{Lemma}
\newtheorem{proposition}[theorem]{Proposition}
\newtheorem{definition}[theorem]{Definition}
\newtheorem{corollary}[theorem]{Corollary}

\newtheorem{condition}{Condition}

\newcommand{\Z}{{\mathbb Z}}

\newcommand{\id}{\operatorname{id}}
\newcommand{\from}{\colon}

\newcommand{\f}{\bar f}
\begin{document}
\title[Cofinite graphs and groupoids]{Cofinite graphs and groupoids and their profinite completions}

\author{Amrita Acharyya}
\address{Department of Mathematics and Statistics\\
University of Toledo, Main Campus\\
Toledo, OH 43606-3390}
\email{Amrita.Acharyya@utoledo.edu}

\author{Jon M. Corson}
\address{Department of Mathematics\\
University of Alabama\\
Tuscaloosa, AL 35487-0350}
\email{jcorson@ua.edu}

\author{Bikash Das}
\address{Department of Mathematics\\
University of North Georgia, Gainesville Campus\\
Oakwood, Ga. 30566}
\email{Bikash.Das@ung.edu}



\begin{abstract}
We define cofinite graphs and cofinite groupoids in a unified way that extends the notion of  cofinite groups introduced by Hartley. 
The common underlying structure of all these objects is that they are directed graphs endowed with a certain type of uniform structure, that we call a cofinite uniformity. 
We begin by exploring the fundamental theory of cofinite directed graphs in full generality. The general theory turns out to be almost completely analogous to that of cofinite groups. 
For instance, every cofinite directed graph has a unique completion which is a compact cofinite direct graph. Moreover, compact cofinite directed graphs are precisely the profinite directed graphs, i.e., projective limits of finite discrete topological directed graphs.  
We then apply the general theory to directed graphs with additional structure such as graphs (in the sense of Serre) and groupoids, thus leading to the notions of cofinite graphs and cofinite groupoids. 
Cofinite groupoids with only finitely many identities behave much the same as cofinite groups, which are the same thing as cofinite groupoids with a single identity. However the situation for cofinite groupoids with infinitely many identities is more complicated.  
\end{abstract}

\maketitle


\section{Introduction} \label{s:Intro}

Embedding an algebraic object, such as a group, ring, or module, into a projective limit of well-behaved objects is a frequently used tactic in algebra and number theory. 
For example, in commutative algebra, polynomial rings are embedded in rings of formal power series. Or more generally, if $R$ is any commutative ring and $I$ is an ideal, then the $I$-adic completion of $R$ is the projective limit of the inverse system of quotient rings $R/I^n$, $n\ge0$. 
The case of $R=\Z$ and $I=(p)$, where $p$ is a prime, yields the $p$-adic integers. These rings are of particular importance in number theory and are instances of projective limits of finite rings, and thus are profinite rings. Similarly, of importance in group theory, any residually  finite group can be embedded in a profinite group (i.e., projective limit of finite groups). 

There is a topological approach to producing such projective limits know as {\it completion\/}.
The rough idea is to impose a suitable topology on the object making it into a topological group (or ring, or  module, etc.) so that Cauchy sequences (or more generally, Cauchy nets) can be defined and used to construct the completion.  By choosing different such topologies, various completions are formed. In the case of a residually finite group, 
Hartley~\cite{bH77} introduced the terminology of cofinite groups for the topological group structures that can be imposed such that the completion is a profinite group. 
Namely, a {\it cofinite group\/} is a topological group $G$ whose open normal subgroups of finite index form a neighborhood base for the identity. 
Our purpose in this paper is to investigate to what extent the technique of profinite completion can be extended to more general structures that may only have at most a partial multiplication; for instance, we consider graphs (in the sense of Serre~\cite{jS80}) and groupoids.

One feature that all the objects that will concern us here have in common is an underlying directed graph structure. (It may be quite simple, as in the case of a group, which we view as a directed graph whose vertex set consists of only the identity element.) 
Thus, we begin by exploring embeddings of directed graphs as dense sub-directed graphs of profinite directed graphs with the goal of developing a general theory that can be applied widely in many situations. 
Initially we note that, without some modification, the topological approach used in the classical situations to construct and distinguish various completions breaks down for directed graphs in general. 
The following easy example illustrates this point. 

\subsubsection*{Example} \label{e:1,2-point compactification} 
Let $\Gamma$ be the directed graph obtained by subdividing the real line at the integer points, so $V(\Gamma)=\Z$, and direct each edge in the increasing direction. 
Incidentally, $\Gamma$ is the Cayley directed graph of the additive group $\Z$ with respect to the generating set $\{1\}$. 
We view $\Gamma$ as a topological directed graph with the discrete topology. 
(Our terminology and notation for directed graphs is fully explained in the next section.)

For each integer $n\ge0$, let $[-n,n]$ denote the connected sub-directed graph of $\Gamma$ with vertex set $\{0,\pm1,\ldots,\pm n\}$. Thus $[-n,n]$ is a combinatorial arc of length $2n$.  
There is a unique retraction $\theta_n\from\Gamma\to[-n,n]$; it is the simplicial map defined on vertices by $\theta(i)=i$ for $-n\le i\le n$, 
$\theta(i)=n$ for $i\ge n$, and $\theta(i)=-n$ for $i\le-n$.
Also form the quotient directed graph $\Gamma_n=[-n,n]/(-n\!=\!n)$ obtained by identifying the vertex $-n$ with the vertex $n$, and let $q_n\from[-n,n]\to\Gamma_n$ denote the natural quotient map. Thus $\Gamma_n$ is a combinatorial circle of length $2n$. Let $\phi_n=q_n\theta_n\from\Gamma\to\Gamma_n$ denote the composite map of directed graphs.  

For integers $0\le m\le n$, let $\theta_{mn}\from[-n,n]\to[-m,m]$ be the restriction of the retraction $\theta_m\from\Gamma\to[-m,m]$, and note that $\theta_{mn}\theta_n=\theta_m$. Also note that $\theta_{mn}$ determines a map of directed graphs $\phi_{mn}\from\Gamma_n\to\Gamma_m$ on the quotient directed graphs, and 
$\phi_{mn}\phi_n=\phi_m$ holds. We have constructed two inverse systems of finite directed graphs and maps of directed graphs indexed by the directed set of nonnegative integers: 
$([-n,n],\theta_{mn})$ and $(\Gamma_n,\phi_{mn})$. Denote their inverse limits by
$\widehat\Gamma_1=\varprojlim\Gamma_n$ and $\widehat\Gamma_2=\varprojlim[-n,n]$.
Endow each of the finite graphs $[-n,n]$ and $\Gamma_n$ with the discrete topology so that  
$\widehat\Gamma_1$ and $\widehat\Gamma_2$ are compact Hausdorff topological directed graphs.

Embed $\Gamma$ in $\widehat\Gamma_1$ via the canonical map of directed graphs determined by the compatible family of maps $\phi_n\from\Gamma\to\Gamma_n$. 
Then it is easy to see that $\widehat\Gamma_1=\Gamma\cup\{\infty\}$ and its underlying topological space is the one-point compactification of the discrete underlying topological space of $\Gamma$. 
Similarly, $\Gamma$ embeds in $\widehat\Gamma_2$ as a dense discrete sub-directed graph and $\widehat\Gamma_2=\Gamma\cup\{-\infty,\infty\}$ is a two-point compactification of $\Gamma$. Thus $\widehat\Gamma_1$ and $\widehat\Gamma_2$ are non-isomorphic topological directed graphs, each containing $\Gamma$ as a dense sub-directed topological directed graph.  
\bigbreak

Thus endowing a directed graph $\Gamma$ with a topological directed graph structure is not enough to uniquely determine a completion of $\Gamma$, and so something more is needed to define a suitable notion of a cofinite directed graph.  
Although, in the  previous example, the relative topology that $\Gamma$ inherits from each $\widehat\Gamma_i$ is the same, namely the discrete topology, the uniform structures they induce on $\Gamma$ are different. 

A profinite directed graph, like any compact Hausdorff space, has a unique uniformity that induces its topology. We explore this uniform structure on  
a profinite directed graph in Section~\ref{s:Profinite},  and characterize it in terms of what we call cofinite entourages (Theorem~\ref{t:Profinite directed graph}). 
Then by analyzing the relative uniform structure induced on a sub-directed graph of a profinite directed graph, a natural notion of cofinite directed graph is discovered 
(Definition~\ref{d:Cofinite directed graph}). As justification for our definition, in Section~\ref{s:Completions} we show that every cofinite graph $\Gamma$ has a unique completion $\overline\Gamma$ (Theorem~\ref{t:Existence and uniqueness}), and it is a profinite directed graph (Theorem~\ref{t:Completion is profinite}). Moreover, profinite directed graphs are precisely the compact cofinite directed graphs.  

In Section~\ref{s:Applications} we conclude that paper with some applications that we are interested in for future reference.  Specifically we consider graphs (in the sense of Serre~\cite{jS80}) and groupoids. These objects are directed graphs with additional structure. We require that the cofinite directed graph structures on them also respect the additional structure. In the case of a graph, it turns out that in general the completion of a cofinite graph is a profinite graph. However for a cofinite groupoid, we need to impose a mild restriction to ensure that the completion is a groupoid.  This condition automatically holds for groupoids with only finitely many identities (such as a group, which has a unique identity), and thus the completion of a cofinite groupoid with a finite vertex set is a profinite groupoid (Theorem~\ref{t:Finite vertex set}).


\section{Preliminaries} \label{s:Preliminaries}


\subsection{Co-discrete equivalence relations} \label{ss:Binary relations}  

We do not distinguish between a binary relation and its graph. Thus, by a {\it binary relation\/} on a set $X$ we mean simply a subset $R\subseteq X\times X$. 
Given a binary relation $R$ on a set $X$ and $a\in X$, we write 
$R[a]=\{x\in X\mid (a,x)\in R\}$. More generally, for any subset $A$ of $X$, we put
$R[A]=\bigcup\{R[a]\mid a\in A\}$.

The following lemma identifies certain equivalence relations on topological spaces that are useful in describing the uniform structure of profinite spaces. Recall that the {\it index\/} of an equivalence relation $R$ on a set $X$ is the cardinality of the set $X/R$ of equivalence classes of $X$ by $R$. 

\begin{lemma} \label{l:Co-discrete}
Let $R$ be an equivalence relation on a topological space~$X$. Then the following conditions are equivalent: 
\begin{enumerate}
\item[(1)] $R$ is an open subset of the product space $X\times X$.
\item[(2)] Every equivalence class $R[a]$ is an open subset of $X$. 
\item[(3)] The quotient space $X/R$ is a discrete space.
\end{enumerate}
Additionally when $X$ is compact, such an equivalence relation $R$ has finite index and is a clopen subset of $X\times X$. 
\end{lemma}

\begin{proof}
By the corollary to Proposition 4 of~\cite[Section I.4.2]{nB66}, $(1)\Rightarrow(2)$. 
Being an equivalence relation,   
$$
R=\bigcup_{a\in X}\big(R[a]\times R[a]\big)
$$
which makes it clear that $(2)\Rightarrow(1)$. 
Note also that $(2)\Leftrightarrow(3)$ since points are open in $X/R$ if and only if equivalence classes are open in $X$. 
To show the last part, assume now that $X$ is compact and that $R$ is an equivalence relation on $X$ satisfying the equivalent conditions~(1)--(3). Then the quotient space $X/R$ is compact and discrete, and thus a finite space. Consequently $R$ is of finite index and it is closed in $X\times X$ since it is the finite union of the closed sets 
$R[a]\times R[a]$, for $a\in X$. 
\end{proof}


\subsection{Directed graphs} \label{ss:Directed graphs}  

A {\it directed graph\/} is a set $\Gamma$ with a distinguished subset $V(\Gamma)$ and two functions $s,t\from\Gamma\to V(\Gamma)$ satisfying $s(x)=t(x)=x$ for all $x\in V(\Gamma)$. 
Each element of $V(\Gamma)$ is called a {\it vertex\/}. 
We denote the complement of $V(\Gamma)$ by 
$E(\Gamma)=\Gamma\setminus V(\Gamma)$ and each $e\in E(\Gamma)$ is called an {\it edge\/} with {\it source\/} $s(e)$ and {\it target\/} $t(e)$. 
Observe that the vertex set $V(\Gamma)$ is determined by the source map $s$, and by the target map $t$: it is precisely the set of fixed points of each of these maps.

\subsubsection*{Sub-directed graphs}
A subset $\Sigma$ of a directed graph $\Gamma$ is called a 
{\it sub-directed graph\/} if it contains the source and target of each of its members. 
Under the restrictions of the source and target maps on $\Gamma$, a sub-directed graph $\Sigma$ is itself a directed graph. 
Note that $V(\Sigma)=V(\Gamma)\cap\Sigma$. 

\subsubsection*{Products of directed graphs}
Let $(\Gamma_i)_{i\in I}$ be a family of directed graphs. We make the Cartesian product 
$\Gamma=\prod_{i\in I}\Gamma_i$ into a directed graph in the following way. 
The vertex set of $\Gamma$ is the set $V(\Gamma)=\prod_{i\in I}V(\Gamma_i)$. 
The source and target maps are defined coordinate-wise: 
if $x=(x_i)\in\Gamma$, then $s(x)=(s(x_i))$ and $t(x)=(t(x_i))$. 
The set $\Gamma$ with this directed graph structure is called a {\it product directed graph}.

\subsubsection*{Maps of directed graphs}
A {\it map of directed graphs\/} is a function $f\from\Gamma\to\Delta$ between directed graphs that preserves sources and targets.
Note that if $f\from\Gamma\to\Delta$ is a map of directed graphs, 
then $f(V(\Gamma))\subseteq V(\Delta)$. For if $x\in V(\Gamma)$, then 
$f(x)=f(s(x))=s(f(x))$ and hence $f(x)\in V(\Delta)$.
However $f$ does not necessarily map edges to edges. Those maps of directed graphs that do map edges to edges are called {\it rigid\/}. 
Compositions of maps of directed graphs are maps of directed graphs. 
Also note that if $f\from\Gamma\to\Delta$ is a map of directed graphs and $\Sigma$ is a 
sub-directed graph of $\Gamma$, then $f(\Sigma)$ is a sub-directed graph of $\Delta$. 

\subsubsection*{Compatible equivalence relations on directed graphs}
A equivalence relation $R$ on a directed graph $\Gamma$ is {\it compatible\/} with the directed graph structure if $(s(x),s(y)\in R$ and $(t(x),t(y))\in R$ whenever $(x,y)\in R$.
In other words, an equivalence  relation $R$ on $\Gamma$ is compatible if and only if $R$ is a sub-directed graph of the product directed graph $\Gamma\times\Gamma$.

If $f\from\Gamma\to\Delta$ is a map of directed graphs, then its kernel 
$$
\ker f=\{ (x,y)\in\Gamma\times\Gamma\mid f(x)=f(y) \}
$$ 
is clearly a compatible equivalence relation on $\Gamma$. 

On the other hand, if $R$ is any compatible equivalence relation on a directed graph $\Gamma$, then the set of equivalence classes $\Gamma/R$ has a unique directed graph structure such that the natural map $\nu\from\Gamma\to\Gamma/R$ is a map of directed graphs. Its source and target maps are given by $s(R[x])=R[s(x)]$ and $t(R[x])=R[t(x)]$, and $V(\Gamma/R)=\nu(V(\Gamma))$. 
Note that in this case $\ker\nu=R$, and hence, every compatible equivalence relation on a directed graph $\Gamma$ is the kernel of some map of directed graphs with domain $\Gamma$. 

Moreover, we have this {\it first isomorphism theorem\/} for directed graphs: 
Let $f\from\Gamma\to\Delta$ be a map of directed graphs and let $K=\ker f$. Then $K$ is a compatible equivalence relation on\/ $\Gamma$ and there is an injective map of directed graphs $f'\from\Gamma/K\to\Delta$ given by $K[x]\mapsto f(x)$, which is an isomorphism when $f$ is surjective.
$$\begindc{\commdiag}[30]
\obj(0,0)[K]{$\Gamma/K$}
\obj(5,5)[Delta]{$\Delta$}
\obj(0,5)[Gamma]{$\Gamma$}
\mor{Gamma}{K}{$\nu$}[\atright,\solidarrow]
\mor{Gamma}{Delta}{$f$}[\atleft,\solidarrow]
\mor{K}{Delta}{$f'$}[\atright,\dashArrow]
\enddc$$

The intersection of any family of compatible equivalence relations on a directed graph $\Gamma$ is again a compatible equivalence relation on $\Gamma$. 
Thus, given any subset $T\subseteq\Gamma\times\Gamma$, the intersection of all compatible equivalence relations on $\Gamma$ containing $T$ is the unique smallest such equivalence relation on $\Gamma$. It is denoted $T^\sharp$ and called the {\it compatible equivalence relation on $\Gamma$ generated by $T$}. 
Alternatively, $T^\sharp$ can be constructed as follows. Let 
$\langle T\rangle= T\cup(s\times s)[T]\cup(t\times t)[T]$, the smallest sub-directed graph of $\Gamma\times\Gamma$ containing $T$. 
Then $T^\sharp=\langle T\rangle^*$, the reflexive and transitive closure of $\langle T\rangle$.


\section{The uniform structure of a profinite directed graph} \label{s:Profinite}

A {\it topological directed graph\/} is a directed graph $\Gamma$ endowed with a topology such that the source and target functions  
$s,t\from\Gamma\to V(\Gamma)$ are continuous maps onto the subspace $V(\Gamma)$.
A {\it profinite graph\/} is a topological group which is the limit of an inverse system of finite discrete topological graphs, or equivalently, a compact Hausdorff topological graph; see, for instance,~\cite{RZ00}. 
Note that the vertex set $V(\Gamma)$ of a profinite directed graph $\Gamma$ is closed, since it is the fixed point set of a continuous map of a Hausdorff space to itself. On the other hand, the edge set $E(\Gamma)=\Gamma\setminus V(\Gamma)$ need not be closed. 

Recall that every compact Hausdorff space $X$ has a unique uniformity that induces its topology. This uniformity consists of all neighborhoods of the diagonal $\delta_X$ in 
$X\times X$; see, for instance,~\cite[II.4.1, Theorem~1]{nB66}. Therefore, as we may, we regard every compact Hausdorff space as a uniform space endowed with the unique uniformity inducing its topology. Our goal in this section is to characterize profinite directed graphs in terms of their unique uniform structures. 

For this purpose, we make use of the following observation. 

\begin{lemma} \label{l:Profinite space}
Let $X$ be a compact, Hausdorff, totally disconnected space and let $W$ be an entourage of $X$. Then there exists an equivalence relation $R$ on $X$ such that $R$ is open in 
$X\times X$ and $R\subseteq W$. 
In particular, $R$ has finite index and is an entourage of $X$. 
\end{lemma}

\begin{proof}
Let $I$ be the set of all equivalence relations $R$ on $X$ such that $R$ is a clopen subset of  $X\times X$. 
Let $x,y\in X$ with $x\ne y$. Since $X$ is Hausdorff and the clopen subsets of $X$ form a base for its topology (see, for example,~\cite[Lemma~0.1.1(c)]{jW98}), there is a clopen subset $C$ of $X$ with $x\in C$ and $y\notin C$. The equivalence relation $R$ that partitions $X$ into the two equivalence classes $C$ and $C'=X\setminus C$, namely
$$
R=(C\times C) \cup (C'\times C'),
$$
is a clopen subset of $X\times X$ and $(x,y)\notin R$. It follows that the intersection of the family $I$ is equal to the diagonal $\delta_X$ in $X\times X$.

By replacing $W$ with its interior in $X\times X$, we may assume that $W$ is an open neighborhood of $\delta_X$ in $X\times X$. So for each $R\in I$, the set $R\setminus W$ is closed in $X\times X$ and 
$$
\bigcap_{R\in I}R\setminus W\subseteq \bigcap_{R\in I}R\setminus\delta_X=\emptyset.
$$
However $X\times X$ is compact, so there exists finitely many $R_1,\ldots,R_n\in I$ such that $(R_1\setminus W)\cap\cdots\cap(R_n\setminus W)=\emptyset$. 
Let $R=R_1\cap\cdots\cap R_n$. Then $R$ is an equivalence relation on $X$ which is a clopen subset of $X\times X$ and $R\subseteq W$. Furthermore, $R$ is a neighborhood of $\delta_X$ in $X\times X$, and so $R$ is an entourage of $X$. 
\end{proof}

To describe the uniform structure of profinite directed graphs and their uniform sub-directed graphs, it is convenient to make the following definition. 

\begin{definition}[Cofinite entourage] \rm
Let $\Gamma$ be a directed graph endowed with a uniformity. 
A {\it cofinite entourage\/} of $\Gamma$ is a compatible equivalence relation of finite index on $\Gamma$ which is also an entourage of~$\Gamma$. 
\end{definition}

Here is a list of some fundamental properties of cofinite entourages that the reader can easily verify.  

\begin{lemma} \label{l:Properties of cofinite entourages}
Let\/ $\Gamma$ and $\Delta$ be directed graphs endowed with uniformities. 
Then their cofinite entourages have the following properties. 
\begin{enumerate}
\item If $R$ is a cofinite entourage of\/ $\Gamma$ and $x\in\Gamma$, then $R[x]$ is a clopen neighborhood of $x$ in\/ $\Gamma$.
\item Every cofinite entourage of\/ $\Gamma$ is a clopen subset 
of\/ $\Gamma\times\Gamma$.
\item If $R_1$ and $R_2$ are cofinite entourages of\/ $\Gamma$, then  
$R_1\cap R_2$ is also a cofinite entourage of\/ $\Gamma$. 
\item If $R$ is a compatible equivalence relation on\/ $\Gamma$ that contains a cofinite entourage of\/ $\Gamma$, then $R$ is itself a cofinite entourage.
\item If $f\from\Gamma\to\Delta$ is a uniformly continuous map and $S$ is a cofinite entourage of $\Delta$, then $(f\times f)^{-1}[S]$ is a cofinite entourage of\/ $\Gamma$.
\item If\/ $\Sigma$ is a sub-directed graph of\/ $\Gamma$ and $R$ is a cofinite entourage of\/ $\Gamma$, then the restriction $R\cap(\Sigma\times\Sigma)$ is a cofinite entourage on\/ $\Sigma$. 
\end{enumerate}
\end{lemma}

We are now ready to give our characterization of profinite directed graphs in terms of their uniform structures. 

\begin{theorem} \label{t:Profinite directed graph}
Let\/ $\Gamma$ be a directed graph endowed with a compact Hausdorff topology. 
Then\/ $\Gamma$ is a profinite directed graph if and only if the set of cofinite entourages of\/  $\Gamma$ is a base for the uniformity on\/ $\Gamma$. 
\end{theorem}

\begin{proof}
Suppose first that $\Gamma$ is a profinite directed graph and let $W$ be an entourage of $\Gamma$. By Lemma~\ref{l:Profinite space}, there exists an equivalence relation $R$ on $\Gamma$ such that $R$ is open in $X\times X$ and $R\subseteq W$. Let 
$$
S=R\cap(s\times s)^{-1}[R]\cap(t\times t)^{-1}[R].
$$
Then $S$ is an equivalence relation on $\Gamma$ and it is open in $X\times X$. 
We claim that $S$ is also a subgraph of $\Gamma\times\Gamma$, and thus is a compatible equivalence relation on $\Gamma$. 
For if $(x,y)\in S$, then 
\begin{enumerate}
\item[(i)] $s(x,y)=(s(x),s(y))\in R$ since $(x,y)\in(s\times s)^{-1}[R]$,
\item[(ii)] $s(x,y)\in (s\times s)^{-1}[R]$ since $s(s(x),s(y))=(s(x),s(y))\in R$, and
\item[(iii)] $s(x,y)\in (t\times t)^{-1}[R]$ since $t(s(x),s(y))=(s(x),s(y))\in R$.
\end{enumerate}
Hence $s(x,y)\in S$. Similarly $t(x,y)\in S$, and therefore $S$ is compatible. Furthermore, it is of finite index by Lemma~\ref{l:Co-discrete}. Hence $S$ is a cofinite entourage of $\Gamma$ such that $S\subseteq W$. It follows that the set of all cofinite entourages of $\Gamma$ a base for the uniformity on\/ $\Gamma$.

Conversely, assume now that the set, say $I$, of all cofinite entourages of $\Gamma$ is a base for the uniformity on\/ $\Gamma$. 
Then the family $I|_{V(\Gamma)}$, consisting of all $R\cap[V(\Gamma\times\Gamma]$ for  $R\in I$, is a fundamental system of entourages of the uniform vertex space $V(\Gamma)$. Note also that if $R\in I$, then
$(s\times s)[R]\subseteq R\cap[V(\Gamma\times V(\Gamma)]$ and  
$(t\times t)[R]\subseteq R\cap[V(\Gamma\times V(\Gamma)]$
since $R$ is a sub-directed graph of $\Gamma\times\Gamma$. It follows that $s$ and $t$ are (uniformly) continuous, and so $\Gamma$ is a topological directed graph. 
It remains to see that $\Gamma$ is totally disconnected. 
Let $x\in\Gamma$ and let $C$ be the component of $\Gamma$ containing $x$. Then 
$C\subseteq\bigcap_{R\in I}R[x]=\{x\}$ since the $R[x]$ are clopen sets and form a neighborhood base for $\Gamma$ at $x$ and $\Gamma$ is Hausdorff. It follows that $C=\{x\}$, and thus $\Gamma$ is totally disconnected. 
\end{proof}


\section{Cofinite directed graphs} \label{s:Cofinite}

Let $\Gamma$ be a profinite directed graph and let $\Sigma$ be a sub-directed graph endowed with the induced uniformity. Then $\Sigma$ is Hausdorff and the family of all $R\cap(\Sigma\times\Sigma)$ as $R$ runs though all cofinite entourages of $\Gamma$, is a fundamental system of entourages of $\Sigma$. Since each $R\cap(\Sigma\times\Sigma)$ is a cofinite entourage of the uniform sub-directed graph $\Sigma$, it follows that every uniform sub-directed graph $\Sigma$ of a profinite directed graph $\Gamma$ has a fundamental system of entourages consisting of cofinite entourages of $\Sigma$. In light of this observation, we make the following definition.

\begin{definition}[Cofinite directed graph] \label{d:Cofinite directed graph} \rm
A {\it cofinite directed graph\/} is a directed graph $\Delta$ endowed with a Hausdroff uniformity that has a base consisting of cofinite entourages of~$\Delta$.  
\end{definition}

Note that by virtue of Theorem~\ref{t:Profinite directed graph}, profinite directed graphs are precisely the compact cofinite directed graphs. 
Also note that the definition of cofinite directed graphs does not assume that the source and target maps are uniformly continuous, or even continuous. However, we next observe that this is automatically true. 

\begin{lemma} \label{l:Uniformly continuous source and target}
Let\/ $\Gamma$ be a cofinite directed graph. Then the source and target maps $s,t\from\Gamma\to V(\Gamma)$ are uniformly continuous maps onto the uniform subspace $V(\Gamma)$. In particular, $\Gamma$ is a Hausdorff topological graph and $V(\Gamma)$ is a closed subset of\/ $\Gamma$. 
\end{lemma}

\begin{proof}
Let $R$ be a cofinite entourage of $\Gamma$. Then 
$$
(s\times s)[R]\subseteq R\cap[V(\Gamma)\times V(\Gamma)] 
\hbox{~and~} (t\times t)[R]\subseteq R\cap[V(\Gamma)\times V(\Gamma)]
$$
and thus $s$, $t$ are uniformly continuous since the sets 
$R\cap[V(\Gamma)\times V(\Gamma)]$, as $R$ runs through all cofinite entourages of $\Gamma$, form a base for the relative uniformity on $V(\Gamma)$. 
The last part follows since uniformly continuous maps are continuous and the vertex set of a Hausdorff topological directed graph is closed. 
\end{proof}


We always endow a cofinite directed graph with the topology induced by its uniformity, and as we just observed, this makes it into a Hausdorff topological directed graph. 
This topology can be characterized as follows. 

Let $\Gamma$ be a cofinite directed graph and let $I$ be a set of cofinite entourages of $\Gamma$ forming a base for the uniformity on $\Gamma$. 
We denote the closure of a subset $A$ of $\Gamma$ (in its uniform topology) by 
$\overline A$, and note that it is given by 
$$
\overline A=\bigcap_{R\in I}R[A]
$$
and each $R[A]$ is a clopen neighborhood of $A$ in $\Gamma$; 
see, for instance, Proposition~II.1.2.2 and its corollary of~\cite{nB66}. 
Similarly, if $W\subseteq\Gamma\times\Gamma$, then its closure $\overline W$ in the product space $\Gamma\times\Gamma$ is given by 
$$
\overline W=\bigcap_{R\in I}(R\times R)[W]=\bigcap_{R\in I}RWR
$$
and each $RWR$ is a clopen neighborhood of $W$ in $\Gamma\times\Gamma$.
See \cite[Chapter II]{nB66} or \cite[Chapter 6]{jK55} for more on the basic properties of uniform topologies.   

As we noted above, every sub-directed graph of a profinite directed graph is a cofinite directed graph when it is endowed with its relative uniformity. On the other hand, cofinite directed graphs can be constructed from scratch as follows. 

Let $\Gamma$ be any abstract directed graph and let $I$ be a filter base of compatible equivalence relations of finite index on $\Gamma$. (That is, the set $I$ has the property: if $R_1,R_2\in I$, then there exists $R_3\in I$ such that $R_3\subseteq R_1\cap R_2$.) 
Then $I$ is a base for a unique uniformity on $\Gamma$; see for instance~\cite[Theorem~6.2]{jK55}.
Thus $\Gamma$ endowed with this uniformity is a cofinite directed graph provided that the induced topology is Hausdorff. Moreover, it is obvious that all of the various cofinite structures that can be put on $\Gamma$ arise in this way. 

Amongst all uniform structures of this type that can be put on $\Gamma$, there is a unique finest one, namely that in which $I$ consists of {\it all\/} compatible equivalence relations of finite index on $\Gamma$. We next observe that $\Gamma$ endowed with this finest uniformity is Hausdorff, and so is a cofinite directed graph. 

\begin{proposition} \label{p:Finest cofinite uniformity}
Let $\Gamma$ be an abstract directed graph endowed with the unique finest uniformity having a base consisting of compatible equivalence relations of finite index on $\Gamma$. 
Then $\Gamma$ is a discrete cofinite directed graph. 
\end{proposition}

\begin{proof}
Let  $y\in V(\Gamma)$. 
Denote by $\Delta$ the finite (abstract) directed graph with $V(\Delta)=\{a,b\}$ and 
$E(\Delta)=\{e,f,g,\overline g\}$, where $s(e)=t(e)=a$, $s(f)=t(f)=b$, $s(g)=a$, $t(g)=b$, 
$s(\overline g)=b$ and $t(\overline g)=a$. 
There is a unique map of directed graphs $\phi\from\Gamma\to\Delta$ such that $\phi(y)=a$ and $\phi(x)=b$ for all vertices $x\ne y$. The kernel $R=\ker\phi$ is a compatible equivalence relation of finite index on $\Gamma$ such that $R[y]=\{y\}$.  

Now let $y\in E(\Gamma)$ and let $\Delta=\{a,e\}$ be the finite directed  graph with $V(\Delta)=\{a\}$. Define a map of directed graphs $\phi\from\Gamma\to\Delta$ by $\phi(y)=e$ and $\phi(x)=a$ for all $x\ne y$. Then the kernel $R=\ker\phi$ is a compatible equivalence relation of finite index on $\Gamma$ such that $R[y]=\{y\}$.

We see that every singleton subset $\{y\}$ of $\Gamma$ is open in the topology induced by the uniformity on $\Gamma$, and hence this topology is discrete. In particular, $\Gamma$ is Hausdorff, and so a cofinite directed graph. 
\end{proof}

It would be interesting to know which other filter bases $I$ of compatible equivalence relations of finite index on an abstract directed graph $\Gamma$, besides the largest one, make $\Gamma$ into a cofinite directed graph. The following result helps with this.

\begin{proposition} \label{p:Hausdorff directed graph}
If\/ $\Gamma$ is an abstract directed graph and $I$ is a filter base of compatible equivalence relations of finite index on $\Gamma$, then the following conditions on $\Gamma$ endowed with the uniformity generated by $I$ are equivalent:
\begin{enumerate}
\item $\Gamma$ is Hausdorff, and so $\Gamma$ is a cofinite directed graph; 
\item $\Gamma$ is totally disconnected;
\item $\bigcap_{R\in I}R=\delta_\Gamma$, the diagonal in $\Gamma\times\Gamma$.
\end{enumerate}
\end{proposition}

\begin{proof}
$(1)\Rightarrow(2)$: Let $x\in\Gamma$ and let $C$ be the component of $\Gamma$ containing $x$. For all $R\in I$, note that $C\subseteq R[x]$ since $R[x]$ is a clopen set containing $x$. Furthermore, when $\Gamma$ is Hausdorff, $\bigcap R[x]=\{x\}$, and thus $C=\{x\}$ and $\Gamma$ is totally disconnected. 

$(2)\Rightarrow(3)$: Let $x,y\in\Gamma$ with $x\ne y$. When $\Gamma$ is totally disconnected, there exists an open set $U$ of $\Gamma$ with $x\in U$ and $y\notin U$ since otherwise the subset $\{x,y\}$ would be connected. Since $U$ is open, there exists $R\in I$ such that $R[x]\subseteq U$. Then $y\notin R[x]$ so that $(x,y)\notin R$, and it follows that 
$\bigcap R=\delta_\Gamma$. 

$(3)\Rightarrow(1)$: We have that  
$\overline\delta_\Gamma=\bigcap R\delta_\Gamma R=\bigcap R$ since $\delta_\Gamma\subseteq R$ and $R^2=R$ as each $R\in I$ is an equivalence relation. 
Thus, when (3) holds, $\delta_\Gamma$ is closed in $\Gamma\times\Gamma$ and thus $\Gamma$ is Hausdorff. 
\end{proof}

Consequently, we see that every cofinite directed graph $\Gamma$ is totally disconnected and that the intersection of all of its cofinite entourages is equal to the diagonal $\delta_\Gamma$ of $\Gamma\times\Gamma$.


\section{Completions of cofinite directed graphs} \label{s:Completions}

In this section we define completions of cofinite directed graphs analogously to the way that  completions of cofinite groups are defined in~\cite{bH77}. 
We show that every cofinite directed graph $\Gamma$ has a completion $\overline\Gamma$ (using a standard construction), and that it is unique up to an isomorphism extending the identity map on $\Gamma$. 
It turns out to be precisely the Hausdorff completion of its underlying uniform space (Corollary~\ref{c:Hausdorff completion}).
Further generalizing the situation for cofinite groups~\cite{bH77}, we observe that the completion $\overline\Gamma$ is a profinite directed graph and that its cofinite entourages are precisely the closures $\overline R$ in 
$\overline\Gamma\times\overline\Gamma$ of the cofinite entourages $R$ of $\Gamma$. 
We begin with an analogue of~\cite[Theorem~2.1]{bH77}. 

\begin{theorem} \label{t:Completion}
Let\/ $\Gamma$ be a cofinite directed graph which is embedded as a dense uniform sub-directed graph of a compact Hausdorff topological directed graph\/~$\overline\Gamma$. If $\Delta$ is any compact Hausdorff topological 
directed graph and $f\from\Gamma\to\Delta$ is a uniformly continuous map of directed graphs, then there exists a unique continuous map of directed graphs $\f\from\overline\Gamma\to\Delta$ extending $f$. 
\end{theorem}

\begin{proof}
Being compact and Hausdorff, $\Delta$ is a complete uniform space. So there exists a unique uniformly continuous map $\f\from\overline\Gamma\to\Delta$ such that 
$\f|_\Gamma=f$; see, for example,~\cite[II.3.6, Theorem~2]{nB66}. 
It remains to verify that $\f$ is a map of directed graphs. 
Let $B=\{ x\in\overline\Gamma\mid\f(s(x))=s(\f(x)) \}$. 
Then $B$ is closed in $\overline\Gamma$ since $\f\circ s$ and $s\circ\f$ are continuous maps from $\overline\Gamma$ to $\Delta$ and $\Delta$ is Hausdorff. 
However, $\Gamma\subseteq B$ and $\Gamma$ is dense in $\overline\Gamma$. Thus $B=\overline\Gamma$ and hence $\f$ preserves sources. Similarly, $\f$ preserves targets and thus is a map of directed graphs. 
\end{proof}

\begin{corollary} \label{c:Completion}
In the situation of the theorem, $\f(\overline\Gamma) = \overline{f(\Gamma)}$.
\end{corollary}

\begin{proof}
Since $\f$ is continuous, 
$\f(\overline\Gamma)\subseteq\overline{\f(\Gamma)}
=\overline{f(\Gamma)}$.
On the other hand, $f(\Gamma)\subseteq\f(\overline\Gamma)$ and 
$\f(\overline\Gamma)$ is closed, since it is a compact subset of the Hausdorff space $\Delta$. Thus $\overline{f(\Gamma)}\subseteq\f(\overline\Gamma)$.
\end{proof}

Based on the theorem above, we make the following definition. 

\begin{definition}[Completion]\rm
Let $\Gamma$ be a cofinite directed graph. Then any compact Hausdorff topological directed graph $\overline\Gamma$ that contains $\Gamma$ as a dense uniform sub-directed graph is called a {\it completion\/} of $\Gamma$.
\end{definition}

As an immediate consequence of Theorem~\ref{t:Completion}, we have the following result on uniqueness of completions. 

\begin{corollary} \label{c:Uniqueness of completions}
If\/ $\overline\Gamma_1$ and\/ $\overline\Gamma_2$ are completions of a cofinite directed graph\/~$\Gamma$, then there is an isomorphism of topological graphs 
$\overline i\from\overline\Gamma_1\to\overline\Gamma_2$ extending the identity map 
on\/~$\Gamma$.
\end{corollary}

We turn now to the question of existence of completions. 
Let $\Gamma$ be a cofinite directed graph and let $I$ be any cofinal set of cofinite entourages of $\Gamma$. 
Note that $(I,\le)$, where ${\le}={\supseteq}$ is the reverse of inclusion, is a directed set. Indeed $\le$ is a partial order on $I$, and 
since $I$ is cofinal, if $R,S\in I$, then there exists $T\in I$ such that $T\subseteq R\cap S$, and so $T\ge R$ and $T\ge S$. 
 
Now for $R,S\in I$ with $R\le S$, the identity map on $\Gamma$ determines a map of directed graphs $\phi_{RS}\from\Gamma/S\to\Gamma/R$ given by 
$\phi_{RS}(S[x])=R[x]$; it is well defined since $S\subseteq R$ and continuous since $\Gamma/S$ has the discrete topology. 
These maps have the properties: (i)~$\phi_{RR}=\id_{\Gamma/R}$ for all $R\in I$; and (ii)~$\phi_{RS}\phi_{ST}=\phi_{RT}$ for all $R\leq S\leq T$ in $I$. 
Hence we have an inverse system of (finite discrete) topological  directed graphs and uniformly continuous maps of directed graphs $(\Gamma/R,\phi_{RS})$, indexed by the directed set $I$. 

Let $\widehat\Gamma=\varprojlim\Gamma/R$ be the inverse limit of this inverse system and denote the projection maps by $\phi_R\from\widehat\Gamma\to\Gamma/R$.  
Then $\widehat\Gamma$ is a topological directed graph and there is a canonical map of directed graphs  $i\from\Gamma\to\widehat\Gamma$ given by $i(a)=(R[a])_{R\in I}$. 
Moreover, $\widehat\Gamma$ is compact and Hausdorff since it is an inverse limit of compact Hausdorff spaces. Thus, as usual, we regard $\widehat\Gamma$ as a uniform space with the unique uniformity compatible with its topology. 

With this setup, we make the following observation. 

\begin{lemma} \label{l:Existence}
If\/ $\Gamma$ is a cofinite directed graph, then 
the inverse limit\/ $\widehat\Gamma=\varprojlim\Gamma/R$ $(R\in I)$ is a compact Hausdorff topological directed graph and the canonical map $i:\Gamma\to\widehat\Gamma$ is a uniform embedding whose image $i(\Gamma)$ is a dense sub-directed graph of\/ $\widehat\Gamma$.
\end{lemma}

\begin{proof}
Since each $\Gamma/R$ has the discrete uniformity, the sets $\ker\phi_R$, as $R$ runs through $I$, is a fundamental system of entourages of $\widehat\Gamma$. 
Also $(i\times i)^{-1}[\ker\phi_R]=R$ for each $R\in I$. It thus follows that $i\from\Gamma\to\widehat\Gamma$ is uniformly continuous. 

Since $\Gamma$ is Hausdorff, $\ker i=\bigcap_{R\in I}R=\delta_{\Gamma}$ by Proposition~\ref{p:Hausdorff directed graph}. Thus $i$ is injective. 
Furthermore, $(i\times i)[R]=\ker\phi_R\cap[i(\Gamma)\times i(\Gamma)]$ for all $R\in I$. Hence $i$ is a uniform isomorphism from $\Gamma$ onto the uniform subspace $i(\Gamma)$ of $\widehat\Gamma$. That is, $i$ is a uniform embedding. 

It remains to show that $i(\Gamma)$ is dense in $\widehat\Gamma$.  
Let $x\in\widehat\Gamma$ and let $U$ be an open subset of $\widehat\Gamma$ containing $x$. Then $(\ker\phi_R)[x]\subseteq U$ for some $R\in I$. 
But  $\phi_R(x)=R[a]=\phi_R(i(a))$ for some $a\in\Gamma$, and so 
$i(a)\in(\ker\phi_R)[x]\subseteq U$. Thus $U\cap i(\Gamma)\ne\emptyset$, and it follows that $i(\Gamma)$ is dense in $\widehat\Gamma$.
\end{proof}

Combining these observations we obtain the following result. 

\begin{theorem}[Existence and uniqueness of completions] \label{t:Existence and uniqueness}
Every cofinite directed graph\/ $\Gamma$ has a completion, and it is unique up to an isomorphism of topological directed graphs extending the identity map on\/~$\Gamma$.
\end{theorem}

\begin{corollary} \label{c:Hausdorff completion}
If\/ $\Gamma$ is a cofinite directed graph, then its completion\/ $\overline\Gamma$ is equal to the Hausdorff completion of the underlying uniform space of\/ $\Gamma$. 
\end{corollary}

\begin{proof}
Since a compact Hausdorff space (endowed with its unique compatible uniformity) is  complete, 
this follows from~\cite[II.3.6, Theorem~2]{nB66}.
\end{proof}

In the definition of the completion of a cofinite directed graph $\Gamma$, we did not require that $\overline\Gamma$ be a cofinite directed graph. However, it turns out that this is automatically true. 
In proving this and other things, we modify our notation as follows.  

From now on, if $A$ is a subset of a cofinite directed graph $\Gamma$, then $\overline A$ will denote the closure of $A$ in the completion $\overline\Gamma$.  
Thus the closure of $A$ in $\Gamma$ is given by $\overline A\cap\Gamma$. 
Likewise, if $W\subset\Gamma\times\Gamma$, then $\overline W$ will denote the closure of $W$ in the product space $\overline\Gamma\times\overline\Gamma$; 
and thus $\overline W\cap(\Gamma\times\Gamma)$ is its closure in $\Gamma\times\Gamma$. 
(That is, a bar over a set will mean its closure in the largest available space.)
Note that this convention is consistent with denoting the completion of $\Gamma$ by $\overline\Gamma$ since $\Gamma$ is dense in $\overline\Gamma$. 
It is also consistent with denoting the unique extension given by Theorem~\ref{t:Completion} of a uniformly continuous map $f$ from $\Gamma$ to a compact Hausdorff topological directed graph $\Delta$ by 
$\f\from\overline\Gamma\to\Delta$ since $\f$ is the closure in $\overline\Gamma\times\Delta$ of the subset $f\subseteq\Gamma\times\Delta$. 

We first prove a couple of lemmas.

\begin{lemma} \label{l:Completion of a sub-directed graph}
If\/ $\Gamma$ is a cofinite directed graph and\/ $\Sigma$ is a sub-directed graph of $\Gamma$, then\/ $\overline\Sigma$ is a sub-directed graph of\/ $\overline\Gamma$ and 
$V(\overline\Sigma)=\overline{V(\Sigma)}$. In particular, $V(\overline\Gamma)=\overline{V(\Gamma)}$. 
\end{lemma}

\begin{proof}
Since $s(\Sigma)\subseteq\Sigma$ and the source map is continuous, it follows that 
$s(\overline\Sigma)\subseteq\overline\Sigma$. Similarly, $t(\overline\Sigma)\subseteq\overline\Sigma$, and so $\overline\Sigma$ is a subgraph of $\overline\Gamma$. 
Its vertex set $V(\overline\Sigma)=\overline\Sigma\cap V(\overline\Gamma)$ is a closed set containing $V(\Sigma)$, and thus $\overline{V(\Sigma)}\subseteq V(\overline\Sigma)$.
On the other hand, $V(\overline\Sigma)=s(\overline\Sigma)\subseteq\overline{s(\Sigma)}
=\overline{V(\Sigma)}$. Hence $V(\overline\Sigma)=\overline{V(\Sigma)}$.
\end{proof}

It follows from this lemma that if $\Gamma$ is a cofinite directed graph and $\Sigma$ is a sub-directed graph (endowed with the relative uniformity), 
then $\overline\Sigma$ (endowed with the relative topology)  is the completion of $\Sigma$.

\begin{lemma} \label{l:Closure of R}
Let\/ $\Gamma$ be a cofinite directed graph and let\/ $\overline\Gamma$ be its completion. 
If $R$ is a cofinite entourage of\/ $\Gamma$, then its closure $\overline R$ in\/ 
$\overline\Gamma\times\overline\Gamma$ is a cofinite entourage of\/  $\overline\Gamma$ and $\overline R\cap(\Gamma\times\Gamma)=R$.
\end{lemma}

\begin{proof}
The natural map $\nu\from\Gamma\to\Gamma/R$ is a uniformly continuous map of directed graphs, and so by Theorem~\ref{t:Completion}, it extends to a uniformly continuous map of directed graphs $\overline\nu\from\overline\Gamma\to\Gamma/R$. 
Note that $\ker\overline\nu$ is a cofinite entourage of $\overline\Gamma$ since $\Gamma/R$ is discrete. Also note that $\ker\overline\nu\cap(\Gamma\times\Gamma)=\ker\nu=R$. 
To complete the proof, it suffices to show that $\ker\overline\nu=\overline R$.

First note that $R=\ker\nu\subseteq\ker\overline\nu$ and $\ker\overline\nu$ is closed in 
$\overline\Gamma\times\overline\Gamma$. Thus $\overline R\subseteq\ker\overline\nu$. 
Reversely, let $(x,y)\in\ker\overline\nu$ and let $W$ be an open neighborhood of $(x,y)$ in 
$\overline\Gamma\times\overline\Gamma$. Then 
$$
W\cap R=W\cap\ker\nu=W\cap\ker\overline\nu\cap(\Gamma\times\Gamma)\ne\emptyset
$$
since $W\cap\ker\overline\nu$ is an open subset of $\overline\Gamma\times\overline\Gamma$ and $\Gamma\times\Gamma$ is dense in 
$\overline\Gamma\times\overline\Gamma$. 
So $(x,y)\in\overline R$, and thus the opposite inclusion 
$\ker\overline\nu\subseteq\overline R$ also holds. 
\end{proof}

\begin{theorem} \label{t:Completion is profinite}
Let\/ $\Gamma$ be a cofinite directed graph. Then its completion\/ $\overline\Gamma$ is a profinite directed graph and the cofinite entourages of\/ $\overline\Gamma$ are precisely all the $\overline R$, where $R$ runs through all cofinite entourages 
of\/ $\Gamma$.
\end{theorem}

\begin{proof}
Let $U$ be an entourage of $\overline\Gamma$. Choose a closed entourage $W$ such that $W\subseteq U$; see, for example, \cite[II.1.2, Corollary~2]{nB66}. Since $W\cap(\Gamma\times\Gamma)$ is an entourage of $\Gamma$, there exists a cofinite entourage $R$ of $\Gamma$ such that 
$R\subseteq W\cap(\Gamma\times\Gamma)$. By Lemma~\ref{l:Closure of R}, $\overline R$ is a cofinite entourage of $\overline\Gamma$ and $\overline R\subseteq W\subseteq U$ since $W$ is closed. 
It follows that the cofinite entourages of $\overline\Gamma$ form a fundamental system of entourages of $\overline\Gamma$. Therefore $\overline\Gamma$ is a profinite directed graph by Theorem~\ref{t:Profinite directed graph}.

Now let $S$ be any cofinite entourage of $\overline\Gamma$ and put $R=S\cap(\Gamma\times\Gamma)$. 
Then $R$ is a cofinite entourage of $\Gamma$ and it is dense in $S$, since $S$ is an open subset of $\overline\Gamma\times\overline\Gamma$ and $\Gamma\times\Gamma$ is dense in  $\overline\Gamma\times\overline\Gamma$. 
Thus, $S=\overline R\cap S=\overline R$ since $S$ is also closed. 
Conversely, if $R$ is a cofinite entourage of $\Gamma$, then $\overline R$ is a cofinite entourage of $\overline\Gamma$ by Lemma~\ref{l:Closure of R}.
\end{proof}

\begin{corollary} \label{c:Discrete quotients}
If\/ $\Gamma$ is a cofinite directed graph and $R$ is a cofinite entourage of\/ $\Gamma$, then the inclusion map $\Gamma\hookrightarrow\overline\Gamma$ induces an isomorphism of directed graphs $\Gamma/R\to\overline\Gamma/\overline R$ (which is also a homeomorphism).
\end{corollary}

\begin{proof}
Let $f\from\Gamma\to\overline\Gamma/\overline R$ be the map of directed graphs given by $f(x)=\overline R[x]$, i.e., the map induced by the inclusion map $\Gamma\hookrightarrow\overline\Gamma$. Note that $f$ is onto since $\Gamma$ is dense in $\overline\Gamma$, and $\ker f=\overline R\cap(\Gamma\times\Gamma)=R$ by Lemma~\ref{l:Closure of R}. Thus, by the first isomorphism theorem for directed graphs, there is an isomorphism of directed  graphs 
$f'\from\Gamma/R\to\overline\Gamma/\overline R$ such that the following diagram commutes:
$$\begindc{\commdiag}[30]
\obj(0,0)[R]{$\Gamma/R$}
\obj(10,5)[Y]{$\overline\Gamma/\overline R$}
\obj(0,5)[X]{$\Gamma$}
\mor{X}{R}{$q$}[\atright,\solidarrow]
\mor{X}{Y}{$f$}
\mor{R}{Y}{$f'$}[\atright,\dashArrow]
\enddc$$
and $f'$ is also a homeomorphism since $\Gamma/R$ and $\overline\Gamma/\overline R$ are discrete.
\end{proof}


\section{Applications} \label{s:Applications}

\subsection{Cofinite spaces}

A {\it cofinite space\/} is a cofinite directed graph $X$ with $V(X)=X$. In this case, the source and target maps are equal to the identity map on $X$, and so are irrelevant. 
The completion $\overline X$ of a cofinite space $X$ is also a cofinite space since 
$V(\overline X)=\overline{V(X)}=\overline X$ 
(by Lemma~\ref{l:Completion of a sub-directed graph}), and it is compact. Thus $\overline X$ is a profinite space (i.e., a compact Hausdorff totally disconnected space).


\subsection{Cofinite graphs}

By a graph we mean a graph in the sense of Serre~\cite{jS80}. That is, a {\it graph\/} is a directed graph $\Gamma$ equipped with an involution, denoted $x\mapsto x^{-1}$, satisfying the conditions:
$$
\hbox{$s(x^{-1})=t(x)$, $t(x^{-1})=s(x)$, $(x^{-1})^{-1}=x$, 
and $x^{-1}=x\Leftrightarrow x\in V(\Gamma)$}.
$$
Note that if $x$ is an edge of $\Gamma$, then $x^{-1}$ is also an edge and $x^{-1}\ne x$. 
(We call a function on a directed graph satisfying these conditions an {\it involution without fixed edges}.) 
Our terminology and notation for graphs essentially follows that of~\cite{jS83}. 

If $\Gamma$ and $\Delta$ are graphs, a {\it map of graphs\/} $f\from\Gamma\to\Delta$ is a a map of directed graphs that preserves inverses: $f(s(x))=s(f(x))$, $f(t(x))=t(f(x))$, and $f(x^{-1})=f(x)^{-1}$ for all $x\in\Gamma$. 

Let $R$ be an equivalence relation on a graph $\Gamma$. We say that $R$ is a {\it graph equivalence relation\/} if $\Gamma/R$ admits the structure of a graph such that the natural map $\nu\from\Gamma\to\Gamma/R$ is a map of graphs. In this case, it is clear that the graph structure on $\Gamma/R$ is unique and given by: 
$s(R[x])=R[s(x)]$, $t(R[x])=R[t(x)]$, $R[x]^{-1}=R[x^{-1}]$, and 
$V(\Gamma/R)=\nu(V(\Gamma))$. 
It is also easy to see that $R$ is a graph equivalence relation if and only if these two conditions hold:
\begin{itemize}
\item if $(x,y)\in R$, then $(s(x),s(y)),(t(x),t(y)),(x^{-1},y^{-1})\in R$;
\item if $(x,x^{-1})\in R$, then $x\in R[V(\Gamma)]$.
\end{itemize}

\begin{definition}[Cofinite graph] \label{d:Cofinite graph} \rm
Let $\Gamma$ be a graph endowed with a uniformity. A {\it cofinite graph entourage\/} of $\Gamma$ is a graph equivalence relation $R$ of finite index on $\Gamma$  which is also an entourage of $\Gamma$. We say that $\Gamma$ is a {\it cofinite graph\/} if it is Hausdorff and its uniformity has a base consisting of cofinite graph entourages of~$\Gamma$.  
\end{definition}

In particular, a cofinite graph $\Gamma$ is a cofinite directed graph. So by Lemma~\ref{l:Uniformly continuous source and target}, the source and target maps of a cofinite graph $\Gamma$ are uniformly continuous, and by a similar argument, its involution is also uniformly continuous. Thus $\Gamma$ is a Hausdorff (uniform) topological graph. 
Furthermore, by Theorems~\ref{t:Completion is profinite}, its completion $\overline\Gamma$ is a compact cofinite directed graph (i.e., a profinite directed graph).  
We show next that there is a unique way to make $\overline\Gamma$ into a cofinite graph such that $\Gamma$ is a subgraph. 

\begin{theorem} \label{t:Cofinite graph}
If\/ $\Gamma$ is a cofinite graph, then its completion\/ $\overline\Gamma$ has a unique involution making it into a  cofinite graph such that\/ $\Gamma$ is a subgraph. 
Moreover, the cofinite graph entourages of\/ $\overline\Gamma$ are precisely the closures $\overline R$ in\/ $\overline\Gamma\times\overline\Gamma$ of all cofinite graph entourages $R$ of\/ $\Gamma$.
\end{theorem}

\begin{proof}
Applying Theorem~\ref{t:Completion} to the composition $\Gamma\to\Gamma\hookrightarrow\overline\Gamma$ of the involution on $\Gamma$ and the inclusion map, we see that there is a unique continuous extension $\overline\Gamma\to\overline\Gamma$ which we also denote by $x\mapsto x^{-1}$. 
Since $\Gamma$ is dense in $\overline\Gamma$, it is clear that this map on $\overline\Gamma$ inherits the properties: $s(x^{-1})=t(x)$, $t(x^{-1})=s(x)$, and $(x^{-1})^{-1}=x$. 

It remains to show that the fixed set $A=\{x\in\overline\Gamma\mid x=x^{-1}\}$ of this involution is equal to $V(\overline\Gamma)$. 
Note first that $A$ is a closed subset of $\overline\Gamma$ containing $V(\Gamma)$, and thus $V(\overline\Gamma)=\overline{V(\Gamma)}\subseteq A$ 
(using Lemma~\ref{l:Completion of a sub-directed graph}). 
To see the opposite inclusion, let $x\in A$ and let $R$ be a cofinite graph entourage of $\Gamma$. Then the natural map $\nu\from\Gamma\to\Gamma/R$ is a map of graphs, and in particular is involution-preserving. Since $\Gamma$ is dense in $\overline\Gamma$, it is clear that the unique continuous extension $\overline\nu\from\overline\Gamma\to\Gamma/R$ is also involution-preserving.  
Thus $\overline\nu(x)=\overline\nu(x^{-1})=\overline\nu(x)^{-1}$. 
However $\Gamma/R$ is a graph, and so $\overline\nu(x)$ is a vertex of $\Gamma/R$. 
Since $\overline R=\ker\overline\nu$ (by the proof of Lemma~\ref{l:Closure of R}) and the vertex set of $\Gamma/R$ is $\nu(V(\Gamma))$, 
it follows that $x\in\overline R[V(\Gamma)]$. 
However the set of closures $\overline R$ of all cofinite graph entourages $R$ of $\Gamma$ form a base for the uniformity on $\overline\Gamma$. Consequently $x\in\overline{V(\Gamma)}$, and so $x\in V(\overline\Gamma)$ 
by Lemma~\ref{l:Completion of a sub-directed graph}. 

To establish the second part, it suffices (by Theorem~\ref{t:Completion is profinite}) to show that for each cofinite entourage $R$ of $\Gamma$, $R$ is a graph equivalence relation on $\Gamma$ if and only if its closure $\overline R$ is a graph equivalence relation on $\overline\Gamma$.
However, this is an immediate consequence of Corollary~\ref{c:Discrete quotients} since it is clear that the induced isomorphism of directed graphs 
$\Gamma/R\to\overline\Gamma/\overline R$ is also involution-preserving. 
Thus $\nu\from\Gamma\to\Gamma/R$ is a map of graphs if and only if 
$\overline\nu\from\overline\Gamma\to\overline\Gamma/\overline R$ is a map of graphs. 
\end{proof}

\begin{corollary}
If $\Gamma$ is a cofinite graph, then its completion $\overline\Gamma$ is a compact Hausdorff totally disconnected topological graph. 
\end{corollary}

\begin{proof}
By Theorem~\ref{t:Cofinite graph}, $\overline\Gamma$ is a cofinite graph, and thus a topological graph. Moreover, it is compact, Hausdorff, and totally disconnected by Theorem~\ref{t:Completion is profinite}.
\end{proof}

\subsubsection*{Remark}

It is easy to see that a cofinite graph $\Gamma$ is compact if and only if it is isomorphic to the inverse limit of an inverse system of finite discrete topological graphs and maps of graphs. Such topological graphs are called {\it profinite graphs\/}. A somewhat more complicated fact is that profinite graphs are the same thing as compact Hausdorff totally disconnected topological graphs.


\subsection{Cofinite groupoids} 

Above we saw that the completions of cofinite spaces and cofinite graphs are profinite spaces and and profinite graphs, respectively. We will see that the situation for cofinite groupoids is more complicated in general.  

To begin, we review some basic groupoid theory. 
Recall that a groupoid  is a small category in which every arrow is invertible. To be more precise, a {\it groupoid\/} is a directed graph $G$ equipped with a partial binary operation (denoted by juxtaposition) satisfying these conditions: for all $g,h,k\in G$, 
\begin{enumerate} 
\item $gh$ is defined if and only if $t(g)=s(h)$, and in this case $s(gh)=s(g)$ and $t(gh)=t(h)$;
\item if $t(g)=s(h)$ and $t(h)=s(k)$, then $(gh)k=g(hk)$;
\item $s(g)g=g$ and $gt(g)=g$;
\item there exists $g^{-1}\in G$ such that $s(g^{-1})=t(g)$, $t(g^{-1})=s(g)$ and $gg^{-1}=s(g)$, $g^{-1}g=t(g)$.
\end{enumerate}
Customarily the set of composable pairs in a groupoid $G$ is denoted 
$$
G^{(2)}=\{(g,h)\in G\times G\mid t(g)=s(h)\}.
$$
Then the partial binary operation on $G$ is a function $G^{(2)}\to G$, denoted by 
$(g,h)\mapsto gh$.
Likewise, we write $G^{(3)}$ for the set of composable triples in $G$. Then the associativity property (Axiom~2) can be stated as: if $(g,h,k)\in G^{(3)}$, then $(gh)k=g(hk)$.
Also note that there is an identity at each vertex $v$ in a groupoid which we are taking to be $v$ itself; by a standard argument, it is unique. Similarly, the inverse $g^{-1}$ of an element $g$ satisfying Axiom~4 is unique.   

A groupoid $G$ with a single vertex is nothing other than a group, and in this case, $V(G)=\{1\}$ and $G^{(2)}=G\times G$. 

A subset $H$ of a groupoid $G$ is a {\it subgroupoid\/} of $G$ if: (1)~$H$ is a sub-directed graph of $G$; (2)~if $h_1,h_2\in H$ and $(h_1,h_2)\in G^{(2)}$, then $h_1h_2\in H$; 
and (3)~$h^{-1}\in H$ whenever $h\in H$. 
In this case, $H$ is itself a groupoid under the restrictions of the operations on $G$.

Given any groupoid $G$ and vertices $x,y\in V(G)$, we denote by $G(x,y)$ the (possible empty) set of all arrows from $x$ to $y$ in $G$. Note in particular that when $x=y$, the set $G(x,x)$ is a subgroupoid of $G$ with a single vertex, and so is a group. 
The groups $G(x,x)$, for $x\in V(G)$, are called the {\it vertex groups\/} or 
{\it local groups\/} of $G$.  

The {\it product groupoid\/} of groupoids $G_1$ and $G_2$ is the directed graph $G_1\times G_2$ with partial binary operation $(G_1\times G_2)^{(2)}\to G_1\times G_2$ defined as follows. 
If $(g_1,g_2), (h_1,h_2)\in G_1\times G_2$ and $t(g_1,g_2)=s(h_1,h_2)$, then 
$(g_1,g_2)(h_1,h_2)=(g_1h_1,g_2h_2)$. 
The product of an arbitrary family of groupoids is defined similarly. 

If $G$ and $H$ are groupoids, a {\it homomorphism\/} $\phi\from G\to H$ is a function with the property: if $(g_1,g_2)\in G^{(2)}$, then $(\phi(g_1),\phi(g_2))\in H^{(2)}$ and 
$\phi(g_1g_2)=\phi(g_1)\phi(g_2)$. 

An equivalence relation $\rho$ on a groupoid $G$ is called a {\it congruence\/} if it is also a subgroupoid of the product groupoid $G\times G$. That is, a congruence on $G$ is an  equivalence relation $\rho$ on $G$ with the properties: 
(1)~if $(g,h)\in\rho$, then $(s(g),s(h))\in\rho$ and $(t(g),t(h))\in\rho$; 
(2)~if $(g_1,g_2),(h_1,h_2)\in\rho$ and $t(g_1,g_2)=s(h_1,h_2)$, then $(g_1h_1,g_2h_2)\in\rho$; and  
(3)~if $(g,h)\in\rho$, then $(g^{-1},h^{-1})\in\rho$.

Our pattern for defining cofinite objects yields the following definition of a cofinite groupoid.  

\begin{definition} \rm
Let $G$ be a groupoid endowed with a uniformity.  
A {\it cofinite congruence\/} on $G$ is a congruence $\rho$ of finite index on $G$ such that $\rho$ is also an entourage of $G$. 
We say that $G$ is a {\it cofinite groupoid\/} if it is Hausdorff and the set of cofinite congruences on $G$ is a base for its uniformity. 
\end{definition}

Note that a cofinite groupoid is a cofinite directed graph, and thus has uniformly continuous source and  target maps by Lemma~\ref{l:Uniformly continuous source and target}. We observe next that the same goes for the involution and partial product maps. 

\begin{lemma} \label{l:Uniformly continuous involution and partial product}
Let $G$ be a cofinite groupoid. Then the involution map $G\to G$, $g\mapsto g^{-1}$, and the partial product $G^{(2)}\to G$ are uniformly continuous maps. In particular, $G$ is a Hausdorff topological groupoid.  
\end{lemma}

\begin{proof}
Since the cofinite congruences $\rho$ on $G$ form a base for its uniformity, and 
$(g^{-1},h^{-1})\in\rho$ whenever $(g,h)\in\rho$, it follows that the involution map $g\mapsto g^{-1}$ is uniformly continuous. 
Also given any cofinite congruence $\rho$ on $G$, $(\rho\times\rho)\cap G^{(2)}$ is an entourage of $G^{(2)}$ such that: for all $(g_1,h_1), (g_2,h_2)\in G^{(2)}$,  
if $(g_1,g_2)\in\rho$ and $(h_1,h_2)\in\rho$, then $(g_1h_1,g_2h_2)\in\rho$. Thus the partial product $G^{(2)}\to G$ is a uniformly continuous map. 
\end{proof}

It follows immediately from the previous lemma that the vertex groups of a cofinite groupoid are topological groups in their relative topologies. In fact, we can say even more. 

\begin{proposition} \label{p:Vertex groups are cofinite}
If $G$ is a cofinite groupoid and $x\in V(G)$, then the vertex group $G(x,x)$ is a cofinite group. 
\end{proposition}

\begin{proof}
The set $\{\rho[x]\cap G(x,x)\mid\hbox{$\rho$ is a cofinite congruence on $G$}\}$ is a base of open neighborhoods for the identity $x$ in $G(x,x)$. 
However, if $\rho$ is a cofinite congruence on $G$,  then $\rho\cap[G(x,x)\times G(x,x)]$ is a congruence of finite index on the group $G(x,x)$ and $\rho[x]\cap G(x,x)$ is the congruence class of the identity. It follows that each $\rho[x]\cap G(x,x)$ is a normal subgroup of finite index in $G(x,x)$. Hence the identity in $G(x,x)$ has a neighborhood base consisting of open normal subgroups of finite index in $G(x,x)$, and so it is a cofinite group.
\end{proof}

In particular, we see that cofinite groups are precisely the cofinite groupoids with a single vertex. 

\subsubsection*{Questions}
 
Since a cofinite groupoid $G$ is a cofinite directed graph, its completion $\overline G$ is a profinite directed graph by Theorem~\ref{t:Completion is profinite}. Below we give partial answers to these three fundamental questions. 
\begin{enumerate}
\item Under what circumstances is it possible to make the completion $\overline G$ into a topological groupoid such that $G$ is a subgroupoid?
\item When $\overline G$ is a topological groupoid such that $G$ is a subgroupoid and $\rho$ is a cofinite congruence on $G$, is $\overline\rho$ a cofinite congruence on $\overline G$? If so, then $\overline G$ would be a compact cofinite groupoid. 
\item Is every compact cofinite groupoid a profinite groupoid, i.e., the inverse limit of an inverse system of finite discrete topological groupoids? The converse is easily seen to be true. 
\end{enumerate}

To begin, note that by Theorem~\ref{t:Completion}, there is a unique extension of the involution on $G$ to a continuous map $\overline G\to\overline G$ 
which we also denote by $g\mapsto g^{-1}$. 
This map on $\overline G$ inherits the properties: $s(g)g=g$, $gt(g)=g$, $s(g^{-1})=t(g)$, $t(g^{-1})=s(g)$, $gg^{-1}=s(g)$, $g^{-1}g=t(g)$ for all $g\in \overline G$. 
Indeed, the set $A=\{g\in\overline G\mid s(g)g=g\}$ is a closed subset of 
$\overline G$ since it is the set where two continuous functions $\overline G\to\overline G$ agree and $\overline G$ is Hausdorff. 
However $G\subseteq A$ and $G$ is dense in $\overline G$, and so $A=\overline G$ and $s(g)g=g$ for all $g\in\overline G$. Similarly the other properties hold. 

On the other hand, extending the partial product on $G$ to a continuous map 
$\overline{G}^{(2)}\to\overline G$ is not such an easy matter and it may be impossible in general. To get around this, we impose an additional condition on the cofinite groupoid $G$.

\begin{condition} \label{condition1}
$G^{(2)}$ is dense in $\overline G^{(2)}$.
\end{condition}

We endow $G^{(2)}$ with the relative uniformity from $G\times G$ and view it as a cofinite directed graph with source and target maps given by 
$$
\hbox{$s(g,h)=(s(g),s(g))$ and $t(g,h)=(t(h),t(h))$.}
$$ 
Note that $V(G^{(2)})=\{(v,v)\in G^{(2)}\mid v\in V(G)\}$ and that the partial binary operation $G^{(2)}\to G$ is a map of directed graphs. 
Similarly we view $\overline G^{(2)}$ as a cofinite directed graph. 
Then $G^{(2)}$ is a sub-directed graph of $\overline G^{(2)}$, and $\overline G^{(2)}$ is compact as it is a closed subset of $\overline G\times\overline G$. 
Therefore, assuming that $G$ satisfies Condition~\ref{condition1}, $\overline G^{(2)}$ is the completion of $G^{(2)}$, and so by Theorem~\ref{t:Completion}, the partial product $G^{(2)}\to G$ has a unique extension to a continuous map of directed graphs 
$\overline G^{(2)}\to\overline G$. 
It is clear that this partial product on $\overline G$ satisfies the groupoid axioms~(3) and~(4) since they hold on the dense subset $G$ of $\overline G$. Then axiom~(1) also holds since 
$\overline G^{(2)}\to\overline G$ is a map of directed graphs.
However, a further condition seems to be needed to ensure that the associativity axiom~(2) holds. For this purpose, we impose the  following stronger condition on $G$.

\begin{condition} \label{condition2}
$G^{(3)}$ is dense in $\overline G^{(3)}$. 
\end{condition}

We first show that Condition~\ref{condition2} implies Condition~\ref{condition1}. 
Suppose $G$ is a cofinite groupoid satisfying Condition~\ref{condition2} and let 
$(g,h)\in\overline G^{(2)}$. Then $(g,h,t(h))\in\overline G^{(3)}$. So for all cofinite congruences $\rho$ on $G$, there exists $(x,y,z)\in G^{(3)}$ such that 
$(g,x),(h,y),(t(h),z)\in\rho$ (by Condition~\ref{condition2}). 
In particular, $(x,y)\in G^{(2)}$, and it follows that Condition~\ref{condition1} holds.

Next note that the set $B=\{(g,h,k)\in\overline G^{(3)}\mid (gh)k=g(hk)\}$ is a closed subset of $\overline G^{(3)}$ since it is the set where two continuous functions from 
$\overline G^{(3)}\to\overline G$ agree and $\overline G$ is Hausdorff. 
But $G^{(3)}\subseteq B$ because $G$ is a groupoid. So if Condition~\ref{condition2} holds, then $B=\overline G^{(3)}$ and the associativity axiom~(2) holds for $\overline G$. 

We next show that our first two questions have affirmative answers for cofinite groupoids satisfying Condition~\ref{condition2}. 

\begin{theorem} \label{t:Completion of a cofinite groupoid}
Let $G$ be a cofinite groupoid satisfying Condition~\ref{condition2}. Then its completion $\overline G$ is a compact cofinite groupoid containing $G$ as a dense uniform subgroupoid. Furthermore, the cofinite congruences of $\overline G$ consist precisely of the closures $\overline\rho$ of all cofinite congruences $\rho$ on $G$.
\end{theorem}

\begin{proof}
We have seen that $\overline G$ is a compact Hausdorff topological groupoid containing $G$ as a dense uniform subgroupoid. Thus, it suffices to prove the last part since it implies that $\overline G$ is cofinite.

Let $\rho$ be a cofinite congruence on $G$. By Lemma~\ref{l:Closure of R}, $\overline\rho$ is a cofinite entourage of the underlying directed graph of $\overline G$. 
So by Lemma~\ref{l:Properties of cofinite entourages}(2), $\overline\rho$ is an open subset of $\overline G\times\overline G$. Moreover, being compatible with the directed graph structure of $\overline G$, $\overline\rho$ is also a sub-directed graph of $\overline G\times\overline G$.

We claim that $\rho^{(2)}$ is dense in $\overline\rho^{(2)}$. 
Let $x=(x_1,x_2), y=(y_1,y_2)\in\overline\rho$ such that $(x,y)\in\overline\rho^{(2)}$, 
and let $U$, $V$ be open subsets of $\overline G$ such that 
$x\in U$, $y\in V$.
Since $\overline\rho$ is open in $\overline G\times\overline G$, there exists open subsets 
$A_1,A_2,B_1,B_2$ of $\overline G$ such that  
$$
\hbox{
$x\in A_1\times A_2\subseteq U\cap\overline\rho$ 
and $y\in B_1\times B_2\subseteq V\cap\overline\rho$.
}
$$
Note that  $(x_i,y_i)\in(A_i\times B_i)\cap\overline G^{(2)}$ since $t(x_1,x_2)=s(y_1,y_2)$, and so $(A_i\times B_i)\cap\overline G^{(2)}$ is a nonempty open subset of $\overline G^{(2)}$. 
However $G^{(2)}$ is dense in $\overline G^{(2)}$ (since Condition~\ref{condition2} holds), and so there exists some element 
$(g_i,h_i)\in (A_i\times B_i)\cap G^{(2)}$, $i=1,2$.
Now $g=(g_1,g_2)\in A_1\times A_2\subseteq\overline\rho$ so that $g\in\overline\rho\cap(G\times G)=\rho$ (by Lemma~\ref{l:Closure of R}). 
Likewise $h=(h_1,h_2)\in\rho$. Furthermore $t(g)=s(h)$, and so 
$(g,h)\in(U\times V)\cap\rho^{(2)}$. 
It follows that $(x,y)$ is in the closure of $\rho^{(2)}$, and hence $\rho^{(2)}$ is dense in $\overline\rho^{(2)}$, as claimed. 

Denote the partial binary operation on the product groupoid 
$\overline G\times\overline G$ by 
$m\from(\overline G\times\overline G)^{(2)}\to\overline G$.
Being a congruence, $\rho$ is a subgroupoid of $G\times G$ and so $m(\rho^{(2)})\subseteq\rho$. Thus 
$m(\overline\rho^{(2)})=m(\overline{\rho^{(2)}})\subseteq\overline\rho$
since $m$ is continuous and $\rho^{(2)}$ is dense in $\overline\rho^{(2)}$ (by the above claim). It follows that $\overline\rho$ is a subgroupoid of $\overline G\times\overline G$, and hence a congruence on $\overline G$. By Lemma~\ref{l:Closure of R}, $\overline\rho$ is also a cofinite entourage of $\overline G$ and so it is a cofinite congruence on $\overline G$. 

Conversely, let $\sigma$ be any cofinite congruence on $\overline G$. 
Then the restriction $\rho=\sigma\cap(G\times G)$ is a cofinite congruence on $G$, and as in the proof of Lemma~\ref{l:Closure of R}, $\overline\rho=\sigma$.
\end{proof} 

It is also worth noting that under these circumstances, Theorem~\ref{t:Completion} can be strengthened as follows. 

\begin{corollary}  \label{c:Continuous extension is a homomorphism}
If $G$ is a cofinite groupoid satisfying Condition~\ref{condition2}, $H$ is a compact Hausdorff topological groupoid, and $\phi\from G\to H$ is a uniformly continuous homomorphism, then the unique continuous map $\overline\phi\from\overline G\to H$ extending $\phi$ is a homomorphism of groupoids.  
\end{corollary}

\begin{proof}
The set 
$C=\{(g,h)\in\overline G^{(2)}\mid\overline\phi(gh)=\overline\phi(g)\overline\phi(h)\}$ 
is a closed subset of $\overline G^{(2)}$ containing $G^{(2)}$. 
However, Condition~\ref{condition1} holds for $G$, and so $C=\overline G^{(2)}$ and $\overline\phi$ is a homomorphism. 
\end{proof}

Turning to our third question, we see no reason to believe that every compact cofinite groupoid is a profinite groupoid. All we do here is give a sufficient condition for this to hold. We make use of the following notation. Given a map of directed graphs $f\from\Gamma\to\Delta$, we use the same symbol to denote the induced map $f\from\Gamma^{(3)}\to\Delta^{(3)}$. It is defined by $f(x_1,x_2,x_3)=(f(x_1),f(x_2),f(x_3))$. Now consider the following condition on a cofinite groupoid $G$.

\begin{condition} \label{condition3}
The uniformity on $G$ has a base consisting of cofinite congruences $\rho$ whose natural map $\nu\from G\to G/\rho$ has the property that $\nu[G^{(3)}]=(G/\rho)^{(3)}$.
\end{condition}

First we observe that Condition~\ref{condition3} implies Condition~\ref{condition2}.

\begin{lemma} \label{l:condition3 implies condition2}
If $G$ is a cofinite groupoid satisfying Condition~\ref{condition3}, then $G$ also satisfies Condition~\ref{condition2}.
\end{lemma}

\begin{proof}
Let $I$ be a base for the uniformity on $G$ of the type guaranteed by Condition~\ref{condition3}. Given $\rho\in I$, let $\overline\nu\from\overline G\to G/\rho$ be the unique extension of the natural map $\nu\from G\to G/\rho$. 
Then $\overline\nu[\overline G^{(3)}]=(\overline G/\overline\rho)^{(3)}$ also holds.
Since $\overline\rho=\ker\overline\nu$ (by the proof of Lemma~\ref{l:Closure of R}), it follows that $\overline G^{(3)}\subseteq\overline\rho[G^{(3)}]
=\bigcup_{(g,h,k)\in G^{(3)}}(\overline\rho[g]\times\overline\rho[h]\times\overline\rho[k])$.
Using Theorem~\ref{t:Completion is profinite},  we see that the set 
$\{\overline\rho \times\overline\rho\times\overline\rho\mid\rho\in I\}$ is a base for the product uniformity on $\overline G\times\overline G\times\overline G$. 
Thus 
$\overline G^{(3)}\subseteq\bigcap_{\rho\in I}\overline\rho[G^{(3)}]=\overline{G^{(3)}}$. 
It  follows that $G^{(3)}$ is dense in $\overline G^{(3)}$; i.e., Condition~\ref{condition2} holds. 
\end{proof}

\begin{theorem} \label{t:Profinite groupoid}
If $G$ is a cofinite groupoid satisfying Condition~\ref{condition3}, then its completion 
$\overline G$ is a profinite groupoid. 
\end{theorem}

\begin{proof}
Construct $\widehat G$ using the standard construction of Lemma~\ref{l:Existence} with 
$I$ equal to the set of all cofinite congruences $\rho$ on $G$ such that $\nu_\rho[G^{(3)}]=(G/\rho)^{(3)}$, where $\nu_\rho\from G\to G/\rho$ is the natural map. 
By Condition~\ref{condition3}, $I$ is a base for the uniformity on $G$ and hence is indeed a cofinal set of cofinite congruences on $G$. 
Thus $\widehat G=\varprojlim_{\rho\in I} G/\rho$ is isomorphic (as cofinite directed graphs) to the completion of $G$. 
We show that the $G/\rho$, $\rho\in I$, have compatible groupoid structures so that  $\widehat G$ is also a profinite groupoid. 

For $\rho\in I$, define an involution and partial binary operation on $G/\rho$ as follows. 
Let $x,y\in G/\rho$.  Define $x^{-1}=\rho[g^{-1}]$, where $g\in G$ such that $x=\rho[g]$.
If $t(x)=s(y)$, then $(x,y,t(y))\in(G/\rho)^{(3)}$, and so there exists $g,h,k\in G^{(3)}$ such that $x=\rho[g]$, $y=\rho[h]$, and $t(y)=\rho[k]$; we define $xy=\rho[gh]$. These  operations on $G/\rho$ do not depend on the choices of representatives in $G$, since $\rho$ is a congruence on $G$, and thus are well-defined. Furthermore, the axioms of a groupoid carry over directly from $G$ (using the extra property on $\rho$ for associativity). 
Therefore, each $G/\rho$, $\rho\in I$, is a finite discrete topological groupoid. 
Moreover, by the way the operations on the  $G/\rho$ are defined, 
it is clear that if $\rho,\sigma\in I$ with $\rho\le\sigma$, then 
$\phi_{\rho\sigma}\from G/\sigma\to G/\rho$ is a homomorphism of groupoids. 
Thus $\widehat G$ is a profinite groupoid. 

Note that for $\rho,\sigma\in I$ with $\rho\le\sigma$, we have the compatibility condition $\phi_{\rho\sigma}\nu_\sigma=\nu_\rho$. 
Thus the unique continuous extensions $\overline\nu_\rho\from\overline G\to G/\rho$ and 
$\overline\nu_\sigma\from\overline G\to G/\sigma$ also satisfy $\phi_{\rho\sigma}\overline\nu_\sigma=\overline\nu_\rho$ and are homomorphisms of groupoids (by Lemma~\ref{l:condition3 implies condition2} and 
Corollary~\ref{c:Continuous extension is a homomorphism}). Hence, by the universal property of inverse limits, there is a unique homomorphism of groupoids 
$\psi\from\overline G\to\widehat G$ such that $\phi_\rho\psi=\overline\nu_\rho$ for all $\rho\in I$ (where $\phi_\rho\from\widehat G\to G/\rho$ is the canonical projection). 
Since the restriction of $\psi$ to $G$ is the canonical inclusion of $G$ into $\widehat G$, it follows by uniqueness of completions (Corollary~\ref{c:Uniqueness of completions}) that $\psi$ is an isomorphism of the underlying directed graphs. However, a bijective homomorphism of groupoids is an isomorphism of groupoids, and so $\overline G$ is isomorphic to the profinite groupoid $\widehat G$. 
\end{proof}

\begin{corollary}
If $G$ is a compact cofinite groupoid satisfying Condition~\ref{condition3}, then $G$ is a profinite groupoid. 
\end{corollary}

The converse of this corollary is true in general (without assuming Condition~\ref{condition3}). Indeed suppose that $G=\varprojlim_{i\in I}G_i$ is the inverse limit of an inverse system of finite discrete topological groupoids $G_i$ indexed by a directed set $I$. Let $\phi_i\from G\to G_i$ be the canonical projection, for $i\in I$. 
Since each $G_i$ is finite and discrete, 
$\rho_i=\ker\phi_i=(\phi_i\times\phi_i)^{-1}(\delta_{G_i})$ is a cofinite congruence on $G$ and the set $\{\rho_i\mid i\in I\}$, is a base for the unique uniformity on $G$ (a compact Hausdorff space). Hence $G$ is a compact cofinite groupoid.


\subsection{Cofinite groupoids with finite vertex sets}

We will see that cofinite groupoids with only finitely many vertices behave very much like cofinite groups in many regards. To start we make a convenient definition. 

We say that a congruence $\rho$ on a groupoid $G$ is {\it rigid\/} if its restriction to $V(G)$ is trivial; i.e., $\rho\cap[V(G)\times V(G)]=\delta_{V(G)}$. In this case, if $(g,h)\in\rho$, then $s(g)=s(h)$ and $t(g)=t(h)$. 
In particular, when $\rho$ is rigid, the natural map $\nu\from G\to G/\rho$ has the property that $\nu[(G^{(3)}]=(G/\rho)^{(3)}$, 
and so (as in the proof of Theorem~\ref{t:Profinite groupoid}) $G/\rho$ has a unique groupoid structure such that $\nu\from G\to G/\rho$ is a homomorphism (and it is bijective on the vertex sets). 

There is an equivalent alternative approach to rigid congruences using normal subgroups. Notice that if $\rho$ is a rigid congruence on a groupoid $G$, then for each $x\in V(G)$, the congruence class $\rho[x]$ is a normal subgroup of the vertex group $G(x,x)$ and this family $(\rho[x])_{x\in V(G)}$ is {\it coherent\/} in the following sense: if $g\in G$, then 
$\rho[s(g)]^g=\rho[t(g)]$ (where $\rho[s(g)]^g=g^{-1}\rho[s(g)]g$). 
Conversely, if $(N_x)_{x\in V(G)}$ is any coherent family of normal subgroups of $G$, then there is a unique rigid congruence $\rho$ on $G$ such that $\rho[x]=N_x$ for all $x\in V(G)$; it is defined by $(g,h)\in\rho$ if and only if $s(g)=s(h)$, $t(g)=t(h)$, and 
$gh^{-1}\in N_{s(g)}$ (or equivalently, $g^{-1}h\in N_{t(g)}$). 

In order for a groupoid to have rigid congruences of finite index, its vertex set must be finite. And in this case, we make the following very useful observation. 

\begin{lemma} \label{l:Rigid base}
Let $G$ be a cofinite groupoid with a finite vertex set. Then the set of rigid cofinite congruences on $G$ is a base for its uniformity. 
\end{lemma}

\begin{proof}
Since $G$ is Hausdorff, its finite vertex set $V(G)$ is a discrete subset of $G$. 
So for each $x\in V(G)$, there exists a cofinite congruence $\rho_x$ on $G$ such that 
$\rho_x[x]\cap V(G)=\{x\}$. 
Then the finite intersection $\sigma=\cap_{x\in V(\Gamma)}\rho_x$ is a cofinite congruence on $G$ and it is rigid. Consequently, the set of all $\rho\cap\sigma$, where $\rho$ runs through all cofinite congruences on $G$, is a base for the uniformity on $G$ consisting of rigid cofinite congruences.  
\end{proof}

Extending the way cofinite groups are defined, we next observe that every cofinite groupoid with only finitely many vertices is completely determined by its topological groupoid structure. 

\begin{proposition} \label{p:Finite vertex set}
Let $G$ be a topological groupoid with a finite vertex set. Then $G$ has a uniformity making it into a cofinite groupoid and inducing its topology if and only if each $G(x,y)$ is an open subset of $G$ and each vertex group $G(x,x)$ is a cofinite group, for $x,y\in V(G)$. Moreover, in this case, such a uniformity on $G$ is unique. 
\end{proposition}

\begin{proof}
Suppose that the $G(x,y)$ are open and that the $G(x,x)$ are cofinite groups. By taking one component at a time, we may assume that the underlying directed graph of $G$ is connected; i.e., that $G(x,y)$ is nonempty for all $x,y\in V(G)$. Fix $x\in V(G)$. Then for each open normal subgroup $N$ of $G(x,x)$, we claim that there is a unique rigid congruence of finite index $\rho_N$ on $G$ such that $\rho_N[x]=N$. We construct $\rho_N$ by forming a coherent family of normal subgroups $N_y\triangleleft G(y,y)$, $y\in V(G)$. 
For each $y\in V(G)$, choose $g\in G(x,y)$ and put $N_y=N^g$. Note that $N_y$ does not depend on the choice of $g$ since $N$ is a normal subgroup of $G(x,x)$. Let $\rho_N$ be the corresponding rigid congruence on $G$ with $\rho[y]=N_y$ for each vertex $y$. In particular, $\rho[x]=N_x=N$ and it is clear that $\rho_N$ is the unique such congruence on $G$. Furthermore, $\rho_N$ has finite index. 
Indeed, if we make a choice of $g_y\in G(x,y)$ for each $y\in V(G)$, then every congruence class is of the form $\rho_N[g_y^{-1}gg_z]=g_y^{-1}\rho_N[g]g_z=g_y^{-1}(gN)g_z$ for some $y,z\in V(G)$ and $g\in G(x,x)$. Thus, if $G$ has $n$ vertices, then the index of $\rho_N$ is $n^2$ times the index of $N$ in $G(x,x)$. For later use, also note that each congruence class of $G$ by $\rho_N$ is an open subset of $G$ since the cosets $gN$ are open in $G(x,x)$ and we are assuming that each $G(x,y)$ is open in $G$.   

Note that $B=\{\rho_N\mid\hbox{$N$ is an open normal subgroup of $G(x,x)$}\}$ is a base for a uniformity on $G$. Indeed, if $N_1$ and $N_2$ are open normal subgroups of $G(x,x)$, then so is $N_1\cap N_2$ and $\rho_{N_1\cap N_2}=\rho_{N_1}\cap\rho_{N_2}$. We claim that this uniformity induces the topology on $G$. Let $U$ be an open subset of $G$ and let $g\in U$. 
Put $y=s(g)$ and $z=t(g)$ so that $g\in G(y,z)$. Choose $g_y\in G(x,y)$ and $g_z\in G(x,z)$. Then $g_ygg_z^{-1}\in G(x,x)$. Since we are assuming that $G(y,z)$ is open in $G$, we may assume that $U\subseteq G(y,z)$ and so $g_yUg_z^{-1}$ is an open subset of $G(x,x)$ containing $g_ygg_z^{-1}$. However $G(x,x)$ is a cofinite group, and so there exists an open normal subgroup $N$ such that $g_ygg_z^{-1}N\subseteq g_yUg_z^{-1}$. 
Since $g_ygg_z^{-1}N=\rho_N[g_ygg_z^{-1}]=g_y\rho_N[g]g_z^{-1}$, it follows that 
$\rho_N[g]\subseteq U$ and hence $U$ is open in the uniform topology generated by our base~$B$. 
On the other hand, as previously noted, every congruence class $\rho_N[g]$, where $N$ is an open normal subgroup of $G(x,x)$ and $g\in G$, is open in $G$.  Hence the uniform topology generated  by $B$ is equal to the topology on $G$, as claimed.

We have seen that the topology on $G$ is generated by a uniformity with a base consisting of congruences of finite index on $G$. It remains to check that $G$ is Hausdorff. But this follows easily by noting that $G$ is the disjoint union of the open subsets $G(y,z)$, $y,z\in V(G)$, and each $G(y,z)$ is homeomorphic to $G(x,x)$ which is Hausdorff (being a cofinite group). 

Conversely assume that $G$ has a cofinite groupoid structure. 
By Proposition~\ref{p:Vertex groups are cofinite}, each vertex group $G(x,x)$ is a cofinite group. By Lemma~\ref{l:Rigid base}, there exists a rigid cofinite congruence $\sigma$ on $G$. Then for all $g\in G(x,y)$, we see that $\sigma[g]\subseteq G(x,y)$, and so each $G(x,y)$ is open in $G$.

For the last part, consider any cofinite groupoid structure on $G$ that induces its topology. 
By Lemma~\ref{l:Rigid base},  this uniformity on $G$ has a base, say $C$, consisting of rigid congruences of finite index on $G$. Given any $\rho\in C$ and $x\in V(G)$, the congruence class $N=\rho[x]$ is an open normal subgroup of $G(x,x)$. By the uniqueness of such a rigid congruence, $\rho_N=\rho\subseteq\rho$ and $\rho_N\in B$. On the other hand, given any open normal subgroup $N$ of $G(x,x)$, there exists $\rho\in C$ such that $\rho[x]\subseteq N$ (since $N$ is also open in $G$ and $x\in N$). Now $\rho$ and $\rho_N$ are rigid congruences with $\rho[x]\subseteq\rho_N[x]$, and thus $\rho\subseteq\rho_N$. 
It follows that $C$ and $B$ are bases for the same uniformity on $G$, and hence $G$ has a unique cofinite groupoid structure that induces it topology. 
\end{proof}

Viewing groups as groupoids with a single vertex, we obtain the following result as an immediate corollary. 

\begin{corollary}
A Hausdorff topological group $G$ is a cofinite group if and only if there is a uniformity on $G$ with a base consisting of congruences of finite index on $G$ that induces the topology on $G$. Moreover, in this case, such a uniformity on $G$ is unique. 
\end{corollary}

An immediate consequence of Lemma~\ref{l:Rigid base} is that a cofinite groupoid $G$ with a finite vertex set satisfies Condition~\ref{condition3}, and so also Conditions~\ref{condition1} and~\ref{condition2}. Hence Theorems~\ref{t:Completion of a cofinite groupoid} 
and~\ref{t:Profinite groupoid} can be applied in proving the following result.

\begin{theorem} \label{t:Finite vertex set}
Let $G$ be a cofinite groupoid with a finite vertex set. Then its completion $\overline G$ is a profinite groupoid with $V(\overline G)=V(G)$. 
Moreover, the rigid cofinite congruences on $\overline G$ consist precisely of the closures 
$\overline\rho$ of all rigid cofinite congruences $\rho$ on $G$. 
\end{theorem}

\begin{proof}
By Theorem~\ref{t:Profinite groupoid}, $\overline G$ is a profinite groupoid. 
By Lemma~\ref{l:Completion of a sub-directed graph}, 
$V(\overline G)=\overline{V(G)}=V(G)$ (since a finite subset of a Hausdorff space is closed). 
The remaining part follows from Theorem~\ref{t:Completion of a cofinite groupoid} since when $V(\overline G)=V(G)$, it is clear that a cofinite congruence $\rho$ on $G$ is rigid if and only if its closure $\overline\rho$ is rigid.  
\end{proof}

\bibliographystyle{amsplain}

\end{document}